\documentclass{amsart}
\usepackage[english]{babel}
\usepackage{latexsym,amsfonts}
\usepackage{amssymb}
\usepackage{amsmath}

\newtheorem{lemma}{Lemma}[section]
\newtheorem{theorem}[lemma]{Theorem}
\newtheorem{cor}[lemma]{Corollary}
\newtheorem{proposition}[lemma]{Proposition}

\newcommand{\Alt}[1]{\mathrm{Alt}(#1)}

 \newcommand{\PSL}{\mathrm{PSL}}
 \newcommand{\Sz}{\mathrm{Sz}}
 \newcommand{\M}{\mathrm{M}}
 \newcommand{\McL}{\mathrm{McL}}
 \newcommand{\PSp}{\mathrm{PSp}}

 \newcommand{\SL}{\mathrm{SL}}
 \newcommand{\diag}[1]{\mathrm{diag}({#1})}
 \def\F{\mathbb{F}}
 \def\Z{\mathbb{Z}}

\begin{document}

\title[]{The $(2,3)$-generation of the special linear groups over finite fields}

\author{Marco Antonio Pellegrini}
\address{Dipartimento di Matematica e Fisica, Universit\`a Cattolica del Sacro Cuore, Via
Musei 41,
25121 Brescia, Italy}
\email{marcoantonio.pellegrini@unicatt.it}


\begin{abstract}
We complete the classification of the finite special linear groups $\SL_n(q)$ which are $(2,3)$-generated, i.e., which are generated by an involution and an element of order $3$. This also gives the classification of the finite simple groups $\PSL_n(q)$ which are $(2,3)$-generated.
\end{abstract}

\maketitle

\section{Introduction}

It is a well known fact that every finite simple group can be generated by a pair  of suitable elements: for 
alternating groups this is a classical result of Miller \cite{Mi}, for groups of Lie type it is due to Steinberg 
\cite{St} and for sporadic groups it was proved by Aschbacher and Guralnick \cite{AG}.
A more difficult problem is to find for a finite nonabelian (quasi)simple 
group $G$ the minimum prime $r$, if it exists, such that $G$ is 
$(2, r)$-generated, i.e. such that $G$ can be generated by two elements of respective orders $2$ and $r$. We denote such minimum 
prime $r$ by $\varpi(G)$ (setting $\varpi(G)=\infty$ if $G$ is not $(2,r)$-generated for any prime $r$). 
Since groups generated by two involutions are dihedral, we must have $\varpi(G)\geq 3$. 

Miller himself proved that $\varpi(\Alt{n})=3$ if $n=5$ or $n\geq 9$, while it is easy to verify that 
$\varpi(\Alt{n})=5$ if $n=6,7,8$.
The special linear groups were firstly considered in \cite{T87}, where Tamburini showed that $\varpi(\SL_n(q)) = 3$ for all $n\ge 25$ and all prime power $q$. Woldar \cite{Wo} proved that all simple sporadic groups are $(2,3)$-generated, except for $\M_{11}$, $\M_{22}$, 
$\M_{23}$ and $\McL$, for which  $\varpi(G)=5$.
As proved by L\"ubeck and Malle  \cite{LM}, all simple exceptional groups are $(2,3)$-generated with the only 
exception of the Suzuki groups $\Sz(2^{2m+1})$, for which Suzuki himself \cite{Sz} proved that 
$\varpi(\Sz(2^{2m+1}))=5$.

Hence, we are left to consider the finite simple classical groups.
A key result for such groups is due to Liebeck and Shalev, who proved in \cite{LS} that, apart from the infinite 
families $\PSp_4(2^m)$ and $\PSp_4(3^m)$, all finite simple classical groups are $(2,3)$-generated with a finite number 
of exceptions. So, the problem of finding the exact value of $\varpi(G)$ reduces to classifying the exceptions to the 
Liebeck and Shalev' theorem. 
However, their result relies on probabilistic methods and does not provide any estimates on the number or 
the distribution of such exceptions. 
We remark that King proved in \cite{K} that $\varpi(G)\neq \infty$ for all finite 
simple classical groups $G$, but in general the problem of computing the exact value of $\varpi(G)$  is still wide open 
(see \cite{Isc} for a recent survey on this topic). 

In this paper we consider the projective special linear groups $\PSL_n(q)$. Many authors, such as Di Martino, Macbeath 
Tabakov, Tamburini and Vavilov, already dealt with the problem of the $(2,3)$-generation of $\SL_n(q)$. 
Summarizing their results we have the following list of 
$(2,3)$-generated groups:
\begin{itemize}
\item[(\emph{i})] $\PSL_2(q)$ if $q\neq 9$ (see \cite{Mac});
\item[(\emph{ii})] $\SL_3(q)$ if $q\neq 4$ (see \cite{PT35});
\item[(\emph{iii})] $\SL_4(q)$ if $q\neq 2$ (see \cite{PTV});
\item[(\emph{iv})] $\SL_n(q)$ if $5\leq n\leq 11$ (see \cite{PT35,T6,T7,GG,GGT});
\item[(\emph{v})] $\SL_n(q)$ if $n\geq 13$ (see \cite{T});
\item[(\emph{vi})] $\SL_n(q)$ if $n\geq 5$ and $q\neq 9$ is odd (see \cite{DV1, DV}).
\end{itemize}

\noindent We observe that the $(2,3)$-generation of $\SL_n(q)$ clearly implies the $(2,3)$-generation of $\PSL_n(q)$.

Here, using a constructive approach as in  many of the above papers and in particular the permutational method 
illustrated in \cite{T}, we solve the last remaining case, i.e.  we prove the 
$(2,3)$-generation of $\SL_{12}(q)$, obtaining the following classification.

\begin{theorem}\label{main}
The groups $\PSL_2(q)$ are $(2,3)$-generated for any prime power $q$, except when $q=9$.
The groups $\SL_n(q)$ are $(2,3)$-generated for any prime power $q$ and any integer $n\geq 3$, except  when 
$(n,q)\in \{(3,4), (4,2)\}$.
\end{theorem}

\noindent Observe that $\varpi(G)=5$ if $G\in \{\PSL_2(9)\cong \Alt{6}$, $\SL_3(4)$, $\PSL_3(4)$, $\PSL_4(2)\cong \Alt{8}\}$.
Clearly $\SL_2(q)$ cannot be $(2,r)$-generated when $q$ is odd, as the unique involution is the central one.

Regarding the $(2,3)$-generation of the other finite classical groups, we recall that only partial results are available, mainly concerning small or 
high dimensions, see \cite{P67, unit, PT35, PTV, TW, TWG, TV}.

Finally, we recall that the infinite groups $\PSL_n(\Z)$ are $(2,3)$-generated if and only if either $n=2$ or $n\geq 5$, and that
the groups $\SL_n(\Z)$ are $(2,3)$-generated if and only if $n\geq 5$ (see \cite{T,V1,V2,V3}).

\section{The $(2,3)$-generation of $\SL_{12}(q)$}
 
Let $q=p^a$, where $p$ is a prime and let $\F_q$ be the field of $q$ elements.
Let $V$ be a $12$-dimensional 
$\F_q$-space, that we identify with the row vectors 
of $\F_q^{12}$. 
Let $\mathcal{C}=\{e_1,e_2\ldots,e_{12}\}$ be the canonical basis of $V$. 
For any element $\sigma \in \Alt{\mathcal{C}}$, we write $g=\sigma$ to denote the permutation matrix $g\in \SL_{12}(q)$ corresponding to $\sigma$ with respect to $\mathcal{C}$. 
This allows us to consider $\Alt{\mathcal{C}}$ as a subgroup of $\SL_{12}(q)$.

Now, let 
\begin{equation}\label{eqy}
y=(e_1,e_2,e_3)(e_4,e_5,e_6)(e_7,e_8,e_9)(e_{10},e_{11},e_{12})
\end{equation}
and let $x$ be the matrix, written with respect to $\mathcal{C}$, such that:
\begin{itemize}
\item[(a)] $x$ swaps $e_1$ and $e_8$;
\item[(b)] $e_2x=-e_2$ and $e_5 x =e_5$;
\item[(c)] $x$ swaps $e_{3i}$ and $e_{3i+1}$ for all $1\leq i\leq 3$;
\item[(d)] $x$ acts on $\langle e_{11},e_{12}\rangle$ as the matrix $\begin{pmatrix} 1 & 0 \\ t & -1\end{pmatrix}$ with $t\in \F_q$.
\end{itemize}
Clearly $x$ and $y$ have orders, respectively, $2$ and $3$, and 
\begin{equation}\label{h}
H=\langle x,y\rangle 
\end{equation}
is a subgroup of $\SL_{12}(q)$.

First of all we prove the following.

\begin{lemma}\label{alt}
If $p\neq 5$, then the group $H$ contains $\Alt{\mathcal{C}}$.
\end{lemma}

\begin{proof}
Let $c=[x,y]=x^{-1} y^{-1} x y$ e define $\gamma$ according to the following rule:
\begin{itemize}
\item[(a)] $\gamma=c^{12}$, if $p=2$;
\item[(b)] $\gamma=c^{12p}$, if $p\equiv 1 \pmod{10}$;
\item[(c)] $\gamma=c^{24p}$, if $p\equiv 3 \pmod{10}$;
\item[(d)] $\gamma=c^{6p}$, if $p\equiv 7 \pmod{10}$;
\item[(e)] $\gamma=c^{18p}$, if $p\equiv 9 \pmod{10}$.
\end{itemize}
It is easy to see that $e_1\gamma=-e_3$, $e_3\gamma=e_5$, $e_5\gamma=e_4$, $e_4\gamma=-e_8$ and $e_8\gamma=e_1$.
Furthermore, $e_i\gamma=e_i$ for all $i\in \{2,6,7,9,10,11,12\}$.
Also taking $\delta=\gamma^y$, we define
$$\eta_1 = (\gamma^{4}\delta^{3}\gamma^2\delta^2)^2,\quad 
\eta_2 = (\gamma^4\delta^3\gamma^2\delta^2 \gamma^2\delta^2)^2,\quad
\eta_3  = (\delta\gamma^2\delta\gamma^2\delta\gamma^3\delta^4\gamma^2)^2.$$
% \begin{eqnarray*}
% \eta_1 &=& (\gamma^{4}\delta^{3}\gamma^2\delta^2)^2,\\
% \eta_2 & =& (\gamma^4\delta^3\gamma^2\delta^2 \gamma^2\delta^2)^2,\\
% \eta_3 & =& (\delta\gamma^2\delta\gamma^2\delta\gamma^3\delta^4\gamma^2)^2.
% \end{eqnarray*}
Since 
$$\eta_1=(e_2,e_5)(e_4,e_8), \quad \eta_2=(e_1,e_6)(e_4,e_9),\quad \eta_3=(e_1,e_3)(e_2,e_8)(e_4,e_9)(e_5,e_6),$$
we obtain that
$\langle \eta_1,\eta_2,\eta_3\rangle =\Alt{\Delta}$, where 
$\Delta=\{e_1,e_2,e_3, e_4,e_5,e_6,e_8,e_9\}\subset\mathcal{C}$.
It follows that $\langle \gamma, \delta \rangle $ contains the subgroup $\Alt{\Delta}$ and in particular the element  $g=(e_1,e_4,e_9)$.
Since $g^x=(e_3,e_{10},e_8)$, we conclude that $H$ contains the subgroup $\langle\Alt{\Delta},  g^x, y\rangle 
=\Alt{\mathcal{C}}$.
\end{proof}

The next key ingredient is the following result, which is a particular case of \cite[Lemma 4.1]{T}. 
As usual, $E_{i,j}$ denotes the elementary matrix having $1$ at position $(i,j)$ and $0$ elsewhere.

\begin{lemma}\label{5}
Let $t\neq 0,2$ be such that $\F_q=\F_p(t)$. Then, the normal closure $N$ of the involution
$w=I_5-2E_{5,5}+tE_{5,4}$ under $\Alt{5}$ is $\langle \SL_5(q), 
\diag{-1,1,1,1,1}\rangle$.
\end{lemma}

We can now prove the following proposition that, combined with the known results on the $(2,3)$-generation of $\SL_n(q)$ described in the Introduction, immediately gives Theorem \ref{main}.

\begin{proposition}\label{propNo5}
For all primes $p\neq 5$ and all integers $a\geq 1$, the groups $\SL_{12}(p^a)$ are $(2,3)$-generated. 
\end{proposition}

\begin{proof}
Set $q=p^a$. Let $H=\langle x,y\rangle $ be as in \eqref{h}, where the element $t\in \F_q$ in $x$ is chosen in such a way that 
$t\neq 0,2$ and $\F_p(t)=\F_q$. As already observed, $H\leq \SL_{12}(q)$. So, we have to prove that $\SL_{12}(q)\leq H$.

First, consider the element $g=(e_1,e_8)(e_9,e_{10})\in \Alt{\mathcal{C}}$. Then $w=gx$ acts on $\langle 
e_8,\ldots,e_{12} \rangle$ as 
the involution $I_5-2E_{5,5}+tE_{5,4}$. By Lemma \ref{5}, we get that $\SL_5(q)$ is contained in 
$K=\langle w, \Alt{ {\{e_8,\ldots,e_{12}\}} }\rangle$.  It follows  that $T=\langle K', 
\Alt{\mathcal{C}}\rangle$ is $\SL_{12}(q)$. Since, by Lemma \ref{alt}, $\Alt{\mathcal{C}}$ is a subgroup of $H$ we have $T\leq H$, whence $H=\SL_{12}(q)$.
\end{proof}

For sake of completeness, using the
permutational method we now prove the $(2,3)$-generation of $\SL_{12}(5^a)$ for all $a\geq 1$. \\
Let $\tilde y=y$ be as in \eqref{eqy} and let $\tilde x$ be the 
matrix, written with respect to $\mathcal{C}$, such that:
\begin{itemize}
\item[(a)] $e_1\tilde x=-e_1$, $e_5 \tilde x =e_5$ and $e_8 \tilde x =e_8$;
\item[(b)] $\tilde x$ swaps $e_{3i}$ and $e_{3i+1}$ for  $i=2,3$;
\item[(c)] $\tilde x$ acts on $\langle e_{2},e_{3},e_4\rangle$ as the involution $x_3=\begin{pmatrix} 
3 & 3 & 2 \\ 2 & 3 & 1\\ 3 & 1 & 3\end{pmatrix}$;
\item[(d)] $\tilde x$ acts on $\langle e_{11},e_{12}\rangle$ as the matrix $\begin{pmatrix} 1 & 0 \\ t & 
-1\end{pmatrix}$ with $t\in \F_q$.
\end{itemize}
Also in this case, $\tilde x$ and $\tilde y$ have orders, respectively, $2$ and $3$, and 
\begin{equation}\label{h5}
\widetilde H=\langle \tilde x,\tilde y\rangle 
\end{equation}
is a subgroup of $\SL_{12}(q)$. 

\begin{lemma}\label{alt5}
The group $\widetilde H$ contains $\Alt{\mathcal{C}}$.
\end{lemma}

\begin{proof}
Let $\tilde c=[\tilde x,\tilde y]$ and define $\tilde \gamma=\tilde c^{12}$ and $\tilde \delta=\tilde \gamma^{y^2}$. 
We firstly observe that both $\tilde \gamma$ and $\tilde \delta$ fix the decomposition  $V=\langle e_1,\ldots,e_8\rangle \oplus \langle e_9\rangle\oplus \ldots\oplus \langle e_{12}\rangle$.
Since $\tilde \gamma\tilde\delta^2$ and $\tilde\gamma\tilde\delta\tilde\gamma^3\tilde\delta^3$ have orders, 
respectively, $313$ and $19531$,
% $$\tilde \gamma\tilde\delta^2,\quad  \tilde\gamma\tilde\delta^2\tilde\gamma^2\tilde\delta^3,\quad
% \tilde\gamma\tilde\delta^2\tilde\gamma^2\tilde\delta^4,\quad 
% \tilde\gamma\tilde\delta\tilde\gamma^2\tilde\delta^2\quad \textrm{and}
% \quad \tilde\gamma\tilde\delta\tilde\gamma^3\tilde\delta^3$$ 
% have orders, respectively, $313$, $2^4\cdot 3 \cdot 13$, $3 \cdot 7\cdot 31$, $2^2\cdot 3\cdot 11 \cdot 71 $ and
% $19531$,
we obtain that $K=\langle \tilde\gamma,\tilde\delta\rangle$ coincides with the subgroup $\left\{\left(\begin{array}{c|c}
A & 0 \\ \hline 0 & I_{4}      \end{array}  \right): A \in \SL_8(5)\right\}\cong \SL_8(5)$
(use, for instance, \cite{LPS}).
In particular, $K$ contains the elements $g_1=\diag{1,x_3,I_{8}}$, $g_2=(e_1,e_2,e_3,e_4,e_5,e_6,e_7)$ and 
$g_3=(e_6,e_7,e_8)$.
Now, as  $g_3^{\tilde yg_1\tilde x}=(e_4,e_8,e_{10})
$, we obtain that $\widetilde H$ contains the subgroup $\langle g_2,g_3^{\tilde yg_1\tilde x},\tilde y\rangle=\Alt{\mathcal{C}}$.
\end{proof}

\begin{cor}
For  all integers $a\geq 1$, the groups $\SL_{12}(5^a)$ are $(2,3)$-generated. 
\end{cor}

\begin{proof}
It suffices to repeat the proof of Proposition \ref{propNo5} using $\tilde x,\tilde y,\widetilde H$ instead of 
$x,y,H$, respectively, and
defining $w=g\tilde x$, where $g=(e_6,e_7)(e_9,e_{10})$.
\end{proof}

\end{document}